\documentclass[12pt]{amsart}
\usepackage{amssymb}
\usepackage{amsmath}
\usepackage{latexsym}
\usepackage{fullpage}
\usepackage{paralist}
\newtheorem{theorem}{Theorem}[section]

\newtheorem*{mytheorem:spectrnodesP}{Theorem~\ref{t:fundspec}A}

\newtheorem{lemma}[theorem]{Lemma}
\newtheorem{proposition}[theorem]{Proposition}

\newtheorem{corollary}[theorem]{Corollary}

\newtheorem*{mycorollary:commoneigP}{Corollary~\ref{c:commoneig}A}
\newtheorem*{mycorollary:initnodesP}{Corollary~\ref{c:initnodes}A}
\newtheorem*{mylemma:uniqueevalueP}{Lemma~\ref{l:uniqueevalue}A}
\newtheorem*{mytheorem:distrootsP}{Theorem~\ref{t:distroots}A}
\newtheorem{remark}[theorem]{Remark}
\newtheorem{example}[theorem]{Example}

\def\N{\mathbb{N}}

\def\R{\mathbb{R}}

\def\Img{\operatorname{Im}}

\def\Rp{\R_+}
\def\Rpn{\Rp^n}
\def\Rpnn{\Rp^{n\times n}}
\def\Rnn{\R^{n\times n}}

\def\crit{{\mathcal C}}

\def\mcm{\lambda}
\def\diag{\operatorname{diag}}

\def\assgraph{{\mathcal G}}
\def\redgraph{{\mathcal R}}
\def\intl{\operatorname{Intl}}

\def\gl{\lambda}
\def\ga{\alpha}

\title{On commuting matrices in max algebra and in classical nonnegative algebra}

\author{Ricardo D. Katz}
\address{CONICET. Postal address: Instituto de Matem\'atica ``Beppo Levi'', Universidad Nacional de Rosario,
Avenida Pellegrini 250, 2000 Rosario, Argentina.}
\email{rkatz@fceia.unr.edu.ar}

\author{Hans Schneider}
\address{Department of Mathematics, University of Wisconsin-Madison, Madison,
Wisconsin 53706, USA} \email{hans@math.wisc.edu}

\author{Serge\u{\i} Sergeev}
\address{University of Birmingham,
School of Mathematics, Watson Building, Edgbaston B15 2TT, UK}
\email{sergeevs@maths.bham.ac.uk}

\thanks{The second and third authors were supported by the EPSRC grant RRAH12809.}
\catcode`\@=11
\@namedef{subjclassname@2010}{%
  \textup{2010} Mathematics Subject Classification}
\catcode`\@=12
\subjclass[2010]{primary: 15A80, 15B48, secondary: 15A27, 15A18.
}
\keywords{Tropical algebra, max-plus algebra, max algebra, nonnegative matrices, commuting matrices, common eigenvectors, Perron-Frobenius, Frobenius normal form}
\begin{document}
\begin{abstract}
This paper studies commuting matrices in max algebra and nonnegative
linear algebra. Our starting point is the existence of a common
eigenvector, which directly leads to max analogues of some classical
results for complex matrices. We also investigate Frobenius normal
forms of commuting matrices, particularly when the Perron roots of
the components are distinct. For the case of max algebra, we show
how the intersection of eigencones of commuting matrices can be
described, and we consider connections with Boolean algebra which
enables us to prove that two commuting irreducible matrices in max
algebra have a common eigennode.
\end{abstract}

\maketitle

\section{Introduction}\label{s:introduction}

The study of commuting complex matrices has a long history. As
observed in~\cite{Dra-51}, Cayley considers what appears to be a
generic case of commuting matrices in his famous
memoir~\cite{Cay-58}. Frobenius~\cite{Fro-78,Fro-96} showed that if
$A_i$, $i = 1,\ldots ,r$, are pairwise commuting matrices, then the
eigenvalues $\ga_i^j$, $j = 1,\ldots ,n$, of the matrices $A_i$ may
be ordered so that the eigenvalues of any polynomial
$p(A_1,\ldots,A_r)$ are $p(\ga^j_1, \ldots \ga^j_r)$, $j = 1,\ldots, n$.
Another proof may be found in Schur~\cite{Schu-02}.
Surprisingly, none of these proofs mention eigenvectors.
Frobenius~\cite{Fro-78} also showed that if for given matrices $A,B$ the
equation $AX = XB$ has a nonzero solution, then $A$ and $B$ have a
common eigenvalue. Another well-known result is that pairwise
commuting matrices  have a common eigenvector. We have found no
reference for the first explicit appearance of this property, though
it easily follows from e.g. the canonical form derived by
Weyr~\cite{Wey-90} and his discussion of commuting matrices.
Many generalizations and applications of this result exist,
see~\cite{DDG-51},~\cite{RR-00} or \cite{MS-10}. Several books on
matrix theory, such as~\cite{Gan-59}, contain proofs of the results
stated above.

It is the purpose of this paper to prove analogs of these results
for matrices over two semirings:

1. the semiring of  nonnegative reals under the usual addition, here
called (classical) nonnegative algebra, and

2. the semiring of nonnegative reals with the operation of maximum
playing the role of addition, here called max algebra.

Spectral theory of matrices in nonnegative algebra is usually called
Perron-Frobenius theory after the founders of this topic,
see~\cite{Per-07,Per-07a,Fro-08,Fro-09,Fro-12}.
The basic results are again found in many books on matrix theory,
such as~\cite{Gan-59,HJ-85,Rot-07}.
Commuting nonnegative matrices can be found in~\cite{BP-94},
see Section~\ref{s:classical}, and in~\cite{Rad-99}.
For further information relevant to the present article see~\cite{Sch-86}.

Spectral theory for matrices in max algebra was developed by
Cuninghame-Green~\cite{CG:79} and Gaubert~\cite{Gau-92},
see~\cite{BCG-09,But-10} for recent expositions.

See \cite{DO-10a,DO-10b} for studies of commuting matrices
in more general semirings.

We devote the main sections to properties of
commuting matrices in max algebra.
At the end of this introduction
we give a formal definition of max algebra and make some remarks on
the relation between the two theories. We then review basic max algebra
spectral theory in Section~\ref{s:SpectralProblem} (for those who
are unfamiliar with this topic). We provide a proof that pairwise
commuting matrices have a common eigenvector in
Section~\ref{s:commoneig}. We also derive some immediate
consequences of this theorem, concerning inequalities for Perron
roots and matrix polynomials, and describe the intersection of
principal eigencones by means of the product of spectral projectors.
In Section~\ref{s:fnf}, we investigate Frobenius normal forms of
commuting matrices, showing that in the important special case where
the Perron roots of the components are distinct, the transitive
closures of the associated reduced digraphs coincide.
In Section~\ref{s:bmaxalg} we consider the eigenvector scaling,
which leads us to study commuting matrices in Boolean algebra.
As a result of this study we show that the critical digraphs
of two commuting irreducible matrices in
max algebra share a common node. In Section~\ref{s:classical} it is
indicated that most results in Sections~\ref{s:commoneig}
and~\ref{s:fnf} also hold in nonnegative matrix algebra.
Section~\ref{s:examples} is devoted to numerical examples.

By {\em max algebra} we understand the analogue of linear algebra
developed over the max-times semiring $(\Rp,\oplus,\times)$, which is
the set of nonnegative numbers $\Rp$ equipped with the operations of
``addition'' $a\oplus b:=\max(a,b)$ and ordinary multiplication
$a\times b$. The zero and unity of this semiring coincide with the usual
$0$ and $1$. The operations of the semiring are extended to
nonnegative matrices and vectors in the same way as in conventional
linear algebra. That is, if $A=(a_{ij})$, $B=(b_{ij})$ and
$C=(c_{ij})$ are matrices of compatible sizes with entries in
$\Rp$, we write $C=A\oplus B$ if $c_{ij}=a_{ij}\oplus b_{ij}$ for
all $i,j$ and $C=A\otimes B$ if $c_{ij}=\oplus_{k} a_{ik}
b_{kj}=\max_{k}(a_{ik} b_{kj})$ for all $i,j$. If $A$ is a square
matrix over $\Rp$, then the iterated product
$A\otimes A\otimes \cdots \otimes A$
in which the symbol $A$ appears $k$ times will be
denoted by $A^{k}$.

It is significant that max algebra can be obtained from
nonnegative linear algebra by means of a limit passage called Maslov
dequantization~\cite{Lit-07}:
\begin{equation}
\label{dequant} a\oplus b=\lim_{p\to\infty} a\oplus_p b \enspace ,
\end{equation}
where $a\oplus_p b:=(a^{1/p}+b^{1/p})^p$. Note that
$(\Rp,\oplus_p,\times)$ forms a semiring which is isomorphic to the semiring
$(\Rp,+,\times)$ of nonnegative numbers equipped with the usual addition
and multiplication. Thus one may expect, and this is indeed the
case, that max algebra and nonnegative linear algebra
have many interesting properties in common\footnote{Referring to
max algebra spectral theory, Gaubert~\cite{GP-97} remarks ``The
theory is extremely similar to the well known Perron-Frobenius
theory''.}. For example, the
Frobenius (normal) from of a reducible matrix plays an important
role in the study of reducible matrices in both theories. In view of
the above discussion, it is not surprising that a comparison of
spectral properties of reducible matrices shows that one needs to replace
strict inequalities in classical nonnegative spectral theory by weak
inequalities in max algebraic spectral theory, see~\cite{BCG-09} for
a remark along these lines concerning eigenvectors.

The above notation employing $\oplus$ and $\otimes$  is standard in
max algebra. However, as many results of the present paper are true
both in max algebra and in nonnegative linear algebra, it will be
convenient to write $a+b$ for $\max(a,b)$ when the argument works in
both theories. On the other hand, we emphasize by using the specific
max algebraic notation when this is not the case.

\section{The spectral problem in max algebra}\label{s:SpectralProblem}

We recall here some notation and basic facts about the spectral
problem in max algebra, which we use further in this paper.
See~\cite{BCOQ,But-10,CG:79,HOW-05} for general reference and more
information.

The {\em max algebraic spectral problem} for $A\in\Rp^{n\times n}$
consists in finding {\em eigenvalues} $\alpha \in \Rp$ and nonzero
{\em eigenvectors} $v \in \Rp^n$ such that $A v=\alpha v$ is
satisfied. Observe that the set $V(A,\alpha):=\{v\mid A v=\alpha
v\}$ is a (max) {\em cone} of $\Rp^n$, that is, a subset of $\Rp^n$
closed under (max) addition and (nonnegative) scalar multiplication.
This cone will be called the {\em eigencone} of $A$ associated with
$\alpha$. The set of eigenvalues, which is nonempty like in the usual
Perron-Frobenius theory, is called the {\em spectrum} of $A$ and
denoted $\Lambda(A)$. The largest eigenvalue of $A$ will be denoted
$\lambda(A)$ and called the {\em Perron root} of $A$ (since we want
the same terminology for max algebra and nonnegative matrix
algebra), and the associated eigencone will be called the {\em
principal eigencone} of $A$.

Unlike in the case of classical algebra,
there is an explicit formula for the max algebraic Perron root of
$A=(a_{ij})\in\Rp^{n\times n}$:
\begin{equation}\label{mcm}
\lambda(A)=\bigoplus_{k=1}^n\bigoplus_{i_1,\ldots,i_k}
(a_{i_1i_2} \cdots a_{i_ki_1})^{1/k}.
\end{equation}
This is also known as the maximal cycle (geometric) mean of $A$.

For $A \in \Rpnn$ we construct the associated digraph
$\assgraph=(N,E)$  by setting $N = \{1,\ldots,n\}$ and
letting $(i,j) \in E$ whenever $a_{ij} >0$.
When this digraph contains at least one cycle,
one distinguishes {\em critical cycles},
where the maximum in~\eqref{mcm} is attained. Further,
one constructs the {\em critical digraph} $\crit(A)=(N_c^A, E_c^A)$,
which consists of all the nodes $N_c^A$ and edges
$E_c^A$ of $\assgraph$ on critical cycles.
The nodes in $N_c^A$ will be called {\em critical nodes} or {\em eigennodes}.

The critical digraph is closely related to the series
\begin{equation}\label{kls}
A^*=I\oplus A\oplus A^2\oplus \cdots \enspace ,
\end{equation}
where $I$ is the unit matrix. This series is known to converge if,
and only if, $\lambda(A)\leq 1$,
in which case it is called the {\em Kleene star} of $A$.
If $\lambda(A)\leq 1$, then this series can be truncated:
$A^*=I\oplus A\oplus A^2\oplus\ldots\oplus A^{n-1}$.

For $A\in\Rpnn$ such that $\lambda(A)=1$, the principal eigencone is
the set of max-linear combinations of all columns of $A^*$ with
indices in $N_c^A$:
\begin{equation}\label{princ-eigen}
V(A,1)=\left\{\oplus_{i\in N_c^A} \beta_i
A^*_{\cdot i} \mid \beta_i\in\Rp\right\} \enspace .
\end{equation}
In particular, we have
\begin{equation}\label{fund-eigv}
AA^*_{\cdot i}=A^*_{\cdot i}\enspace ,
\quad A^*_{i\cdot }A=A^*_{i \cdot } \enspace ,\quad \forall i\in N_c^A \enspace .
\end{equation}
Thus, unlike in the usual Perron-Frobenius theory, even if
$A\in\Rpnn$ is irreducible (that is, the associated digraph
is strongly connected), the principal eigencone in max algebra
may contain more than just one ray. However, for irreducible $A$,
$\lambda(A)$ given by~\eqref{mcm} is the only eigenvalue
and every eigenvector is positive,
see Theorem~\ref{t:fundspec} below.

By standard optimal path algorithms, the critical digraph and the
columns of $A^*$ can be computed in $O(n^3)$ operations.
For further details we refer the reader to~\cite{But-10,HOW-05,BCOQ}.

A (max) cone $K\subset \Rp^n$ is said to be finitely generated if it is
the set of max-linear combinations of a finite subset of vectors of
$\Rp^n$. Equivalently, a cone $K\subset \Rp^n$ is finitely generated
if there exists a matrix $X\in \Rp^{n\times r}$, for some $r\in \N$,
such that $K=\Img(X)$, where as usual $\Img(X):=\{Xu\mid u\in
\Rp^r\}$. Observe that if $K$ is not trivial, we may assume that $X$
does not have a null column. By~\eqref{princ-eigen}, it follows that
the principal eigencone is finitely generated. Indeed, this property
holds for any eigencone of $A$, see e.g.~\cite[Theorem~4.1]{BCG-09}.
Therefore, in what follows, for $\alpha \in \Lambda(A)$ we shall
denote by $X^A_\alpha$ any matrix with nonzero columns satisfying
$V(A,\alpha)=\Img(X^A_\alpha)$.

Let us finally mention that like in classical algebra,
any finite intersection of finitely generated
(max) cones is also finitely generated (this property follows
from~\cite{BH-84}, see e.g.~\cite[Theorem~1]{GK09}).

We summarize the main properties that will be used
in this paper in the next proposition.

\begin{proposition}\label{p:allpropert}
In max algebra the following statements hold:
\begin{enumerate}[(i)]
\item Every matrix has an eigenvalue with a corresponding eigenvector;
\item Eigencones are finitely generated;
\item The intersection of two finitely generated (max) cones is finitely generated.
\end{enumerate}
\end{proposition}

Further information on max algebra spectral theory
will be given in Section~\ref{s:fnf}.

\section{Existence of common eigenvectors}\label{s:commoneig}

\subsection{Common eigenvector of two matrices}

In this section on max algebra we prove that two commuting matrices
have a common eigenvector.
With this aim, we shall need the following lemma.

\begin{lemma}\label{l:InvEigenCone}
If $A,B\in\R_+^{n\times n}$ commute, then any eigencone
$V(A,\alpha)$ of $A$ is invariant under $B$ and any eigencone
$V(B,\alpha)$ of $B$ is invariant under $A$.
\end{lemma}
\begin{proof}
Let $v\in V(A,\alpha)$. For $u=Bv$, we have
\begin{equation}
Au=ABv=BAv=\alpha Bv=\alpha u \enspace .
\end{equation}
Therefore, $B(V(A,\alpha )) \subset V(A,\alpha )$.
\end{proof}

Now it is possible to prove the following result,
which relates the eigencones of two commuting matrices.

\begin{theorem}\label{t:commoneig}
If $A,B\in\R_+^{n\times n}$ commute, then for
any eigencone $V(A,\alpha)$ of $A$ there exists an eigencone
$V(B,\mu)$ of $B$ such that $V(A,\alpha)\cap V(B,\mu)$ contains a
nonzero vector.
\end{theorem}

\begin{proof}
Let $V(A,\alpha)=\Img(X^A_\alpha)$ be an eigencone of $A$. Then,
\begin{equation}
AX_{\alpha}^A=\alpha X_{\alpha}^A \enspace ,
\end{equation}
and since by Lemma~\ref{l:InvEigenCone} we have
$B(\Img(X^A_\alpha))\subset \Img(X^A_\alpha)$, there exists
a (nonnegative square) matrix $C$ such that $BX_{\alpha}^A=X_{\alpha}^A C$.
Let $z$ be any eigenvector of $C$,
so that $Cz=\mu z$ and $z\neq 0$, and consider $u=X_{\alpha}^A z$.
Then, $u\neq 0$ (recall that all the columns of $X_{\alpha}^A$ are
nonzero) and we obtain
\[
Au = AX_{\alpha}^A z=\alpha X_{\alpha}^A z=\alpha u
\]
and
\[
Bu =BX_{\alpha}^A z=X_{\alpha}^A C z= \mu X_{\alpha}^A z=\mu u \enspace .
\]
Thus, $u\in V(A,\alpha)\cap V(B,\mu)$.
\end{proof}

As an immediate consequence, we obtain:

\begin{corollary}\label{c:commoneig}
If $A,B\in\R_+^{n\times n}$ commute,
then they have a common eigenvector.
\end{corollary}

We remark that our proof of Theorem~\ref{t:commoneig}
also shows the following result:

\begin{proposition}\label{p:common}
Let $A \in \Rnn_+$ and let $K$ be a (nontrivial) finitely generated cone of
$\Rp^n$. If $AK \subseteq K$, then $A$ has an eigenvector in $K$.
\end{proposition}

\subsection{Common eigenvector of several matrices}

The results above can be generalized to several pairwise commuting
matrices.

\begin{theorem}\label{t:cv-rmats}
Assume the matrices $A_1,\ldots ,A_r \in \Rnn_+$ commute in pairs.
Then, given any eigenvalue $\alpha_i\in\Lambda(A_i)$,
where $i\in \{1,\ldots ,r \}$,
there exist $\alpha_j\in\Lambda(A_j)$ for $j\neq i$
such that $V(A_1,\alpha_1)\cap \cdots \cap V(A_r,\alpha_r )$
contains a nonzero vector.
\end{theorem}

\begin{proof}
The case $r=2$ is precisely Theorem~\ref{t:commoneig}.
So assume that the statement of the theorem holds for $r=k$ and let
$A_1,\ldots ,A_k,A_{k+1}$ be $k+1$ matrices which commute in pairs.

Without loss of generality, assume $\alpha_1 \in \Lambda (A_1)$ is
given. By the induction hypothesis,
there exist $\alpha_j \in \Lambda(A_j)$, for $j=2,\ldots ,k$,
such that $V(A_1,\alpha_1)\cap \cdots \cap V(A_k,\alpha_k )$
contains a nonzero vector. Moreover,
since by Proposition~\ref{p:allpropert} any eigencone is
finitely generated and any finite
intersection of finitely generated max cones is also finitely
generated, there exists a (nonnegative) matrix $X$ such that
$V(A_1,\alpha_1)\cap \cdots \cap V(A_k,\alpha_k)=\Img(X)$.
Note that we may assume, without loss of generality,
that all the columns of $X$ are nonzero
because $\Img(X)$ contains nonzero vectors.

Since $A_i$ and $A_{k+1}$ commute for $i=1,\ldots,k$,
by Lemma~\ref{l:InvEigenCone} it follows that
\[ 
A_{k+1}(V(A_i,\alpha_i))\subseteq V(A_i,\alpha_i)
\]
for $i=1,\ldots ,k$. Therefore,
$A_{k+1}(\Img(X))=A_{k+1}(\cap_{i=1}^kV(A_i,\alpha_i))\subset
\cap_{i=1}^kV(A_i,\alpha_i)=\Img(X)$ and thus there exists a
(nonnegative square) matrix $C$ such that $A_{k+1}X=XC$.

As in the proof of Theorem~\ref{t:commoneig}, let $z$ be any
eigenvector of $C$ so that $Cz=\mu z$,
for some $\mu \in \Lambda (C)$,
and define $u=X z$.
Since $z\neq 0$ and the columns of $X$ are nonzero,
we have $u\neq 0$ and
\[
A_{k+1}u=A_{k+1} X z=X C z=\mu  X z=\mu u \enspace .
\]
Thus, $u\in \Img(X)\cap V(A_{k+1},\mu)=V(A_1,\alpha_1)\cap \cdots
\cap V(A_k,\alpha_k)\cap V(A_{k+1},\mu)$.
\end{proof}

Next we investigate the eigenvalues of polynomials of commuting matrices.
Let us recall that a {\em max polynomial} is obtained by replacing in a
real polynomial (with nonnegative coefficients) the usual sum by the maximum.

\begin{theorem}\label{t:polynomials}
Let $A_1,\ldots,A_r \in \Rpnn$  commute in pairs and let
$p(x_1,\ldots,x_r)$ be a max polynomial. Then,
\begin{enumerate}[(i)]
\item\label{t:p1}
For each $i\in \{ 1,\ldots ,r \} $ and $\alpha_i\in\Lambda(A_i)$,
there exist $\alpha_j\in\Lambda(A_j)$ for all $j\neq i$ such that
$p(\alpha_1,\ldots,\alpha_r)\in \Lambda(p(A_1,\ldots,A_r))$;
\item\label{t:p2}
For each $\lambda\in\Lambda(p(A_1,\ldots,A_r))$ there exist
$\alpha_i\in \Lambda(A_i)$ for all $i=1,\ldots,r$ such that
$\lambda=p(\alpha_1,\ldots,\alpha_r)$.
\end{enumerate}
\end{theorem}

\begin{proof}
\eqref{t:p1} Let $i\in\{1,\ldots,r\}$ and $\alpha_i\in\Lambda(A_i)$.
By Theorem~\ref{t:cv-rmats}, there exist $\alpha_j\in\Lambda(A_j)$
for all $j\neq i$ and a nonzero vector $v \in \Rpn$ such that
$A_i v=\alpha_i v$ for all $i=1,\ldots, r$. But then we also have
$p(A_1,\ldots,A_r)v=p(\alpha_1,\ldots,\alpha_r)v$, and so
$p(\alpha_1,\ldots,\alpha_r)\in \Lambda(p(A_1,\ldots,A_r))$. \\
\eqref{t:p2} Let $\lambda\in\Lambda(p(A_1,\ldots,A_r))$.
Since $A_1,\ldots,A_r$ and $p(A_1,\ldots,A_r)$ commute in pairs,
by Theorem~\ref{t:cv-rmats} there is an eigenvector
$v\in V(p(A_1,\ldots,A_r),\lambda)$ which is also an eigenvector of
$A_i$ associated with some eigenvalue $\alpha_i\in \Lambda(A_i)$,
for all $i=1,\ldots,r$. But then
$\lambda v=p(A_1,\ldots,A_r) v = p(\alpha_1,\ldots,\alpha_r) v$
and it follows that $\lambda=p(\alpha_1,\ldots,\alpha_r)$.
\end{proof}

\begin{corollary}\label{c:spineqs}
Let $A_1,\ldots, A_r\in\Rpnn$ commute in pairs and
let $p(x_1,\ldots,x_r)$ be a max polynomial. Then,
\begin{enumerate}[(i)]
\item\label{c:spineqs3}
$\lambda(p(A_1,\ldots,A_r))\leq p(\lambda(A_1),\ldots,\lambda(A_r))$;
\item\label{c:spineqs1} $\lambda(A_1+\cdots+ A_r)\leq
\lambda(A_1)+\cdots+\lambda(A_r)$;
\item\label{c:spineqs2} $\lambda(A_1\cdots A_r)\leq
\lambda(A_1)\cdots \lambda(A_r)$.
\end{enumerate}
Moreover, equality holds in all the above relations if the matrices
$A_1,\ldots, A_r$ are irreducible.
\end{corollary}

\begin{proof}
Part~\eqref{c:spineqs3} follows from Theorem~\ref{t:polynomials},
and parts~\eqref{c:spineqs1} and~\eqref{c:spineqs2} are its special cases.
If the matrices are irreducible, then each of them has unique eigenvalue,
and we have the equalities.
\end{proof}

In the case of max algebra we also have
$\lambda(A_1\oplus\ldots\oplus A_r)\geq \lambda(A_i)$ for all
$i=1,\ldots, r$, as the Perron root expressed by~\eqref{mcm} is
monotonic. Hence~\eqref{c:spineqs1} always holds with equality in
max algebra.

\subsection{Intersection of principal eigencones}\label{s:intersection}

A matrix $Q$ is called a {\em projector} on a
cone $K\subset \R^n_+$ if $\Img(Q)=K$ and $Q^2 = Q$.
This implies that $Qx = x$ if, and only if, $x \in K$.
In general, there are many projectors on the same cone,
but if two such projectors $P,Q$ commute,
then they are identical because we have
$Px = QPx = PQx = Qx$ for all $x \in \R^n_+$.

We recall that the eigencone $V(A,\lambda(A))$ associated with
$\lambda(A)$ (assumed to be nonzero) is called the principal
eigencone of $A$ and a projector on $V(A, \lambda(A))$ which
commutes with $A$ is called a {\em spectral projector} for $A$.
Since $V(A,\lambda(A)) = V(A/\lambda(A),1)$,
there is no loss of generality in assuming that $\lambda(A)=1$.
In max algebra, one can explicitly define such projector.
There are two definitions in the literature:
\begin{equation}\label{e:specproj-bac}
\Tilde{Q}(A)=\bigoplus_{i\in N_c^A} A_{\cdot i}^* A_{i\cdot}^* \enspace ,
\end{equation}
and
\begin{equation}\label{e:specproj-yak}
{Q}(A)=\lim_{p\to\infty} \bigoplus_{m\geq p} A^m \enspace .
\end{equation}
The first of these is found in Baccelli et al.~\cite[Section~3.7.3]{BCOQ},
see also~\cite{CTGG-99}, and the second one,
which is attributed to Yakovenko~\cite{Yak-90},
is found in a more general context in
Kolokoltsov and Maslov~\cite[Section~2.4]{KM:97}.

We shall need the following proposition,
which shows that these projectors are indeed identical.
See~\cite[Theorem 2.11]{KM:97} for a closely related result.

\begin{proposition}\label{p:commproj}
Let $A \in \R_+^{n \times n}$ with $\lambda(A) =1$. Then,
there is a unique spectral projector on $V(A,1)$ which is given,
equivalently, by~\eqref{e:specproj-bac} or~\eqref{e:specproj-yak}.
\end{proposition}

\begin{proof}
In the first place, note that in the matrix case,
\eqref{e:specproj-yak} may be replaced by
\begin{gather}
\label{e:specproj-kss} Q(A)=\lim_{p\to\infty} A^p A^* \enspace ,
\end{gather}
see the remarks on $A^*$ in Section~\ref{s:SpectralProblem}.
Since by~\eqref{kls} we have $A^{p+1}A^* \leq A^p A^*$,
it follows that the limit in~\eqref{e:specproj-kss} exists.

By the continuity of operations,
$\lim_{p\to\infty} (C_p B)=(\lim_{p\to\infty} C_p) B$
for any converging sequence of matrices $C_p$ and any matrix $B$.
Using this, we observe that if $B$ is any matrix which commutes with $A$,
then $B$ also commutes with $Q(A)$.
Since as shown above any two commuting projectors
on the same cone are identical,
we conclude that any spectral
projector for $A$ is equal to $Q(A)$.
Therefore, in particular we have $\Tilde{Q}(A)=Q(A)$.
\end{proof}

Next we state two lemmas. The first one exploits~\eqref{e:specproj-yak}
and follows from the continuity of multiplication.
The second lemma is standard and its proof is
recalled for the convenience of the reader.

\begin{lemma}
If $A,B\in\Rnn_+$ commute, then $Q(A)$ and $Q(B)$ commute.
\end{lemma}

\begin{lemma}\label{l:commproj}
Let $Q_i$, $i = 1,\ldots, r$, be commuting projectors. Then,
\begin{equation}
\Img(Q_1) \cap \cdots \cap \Img(Q_r) = \Img(Q_1\cdots Q_r) \enspace .
\end{equation}
\end{lemma}

\begin{proof}
If $x \in \Img(Q_1) \cap \cdots \cap \Img(Q_r)$, then $Q_i x = x$ for
$i = 1,\ldots, r$, and hence $(Q_1\cdots Q_r) x = x$.
Thus, $x \in \Img(Q_1\cdots Q_r)$. Conversely,
if $x \in \Img(Q_1\cdots Q_r)$,
then $(Q_1 \cdots Q_r) y = x$ for some vector $y$.
Multiplying this equation by $Q_i$, for $i =1,\ldots, r$,
using the idempotency of $Q_i$ and commutativity, it follows that $Q_i x = x$.
This completes the proof of the lemma.
\end{proof}

Lemma~\ref{l:commproj} implies that we can express the intersection
of the principal eigencones of commuting matrices as follows:
\begin{equation}\label{qa1ar}
V(A_1,1)\cap\cdots\cap V(A_r,1)=
\Img(Q(A_1)) \cap \cdots \cap \Img(Q(A_r)) =
V(Q(A_1)\cdots Q(A_r), 1)\enspace .
\end{equation}
In the general (reducible) case,
this intersection may reduce to the zero vector.
Since by~\eqref{c:spineqs2} of Corollary~\ref{c:spineqs}
we have $\lambda(Q(A_1)\cdots Q(A_r))\leq 1$,
it follows that~\eqref{qa1ar} is not trivial if,
and only if, the Perron root of $Q(A_1)\cdots Q(A_r)$ is $1$,
in which case this intersection is given by
the principal eigencone of $Q(A_1)\cdots Q(A_r)$.
Using definition~\eqref{e:specproj-bac}, we can compute
this product in $O(rn^3)$ operations, and then it requires no more
than $O(n^3)$ operations to compute its Perron root and describe its
principal eigencone when the Perron root is $1$.

\section{Frobenius normal forms}\label{s:fnf}

Let $\assgraph=(N,E)$ be the associated digraph of $A\in\Rpnn$ and
$\assgraph^\mu =(N^{\mu},E^{\mu})$, for $\mu=1,\ldots ,t$, be the
connected components of $\assgraph$. We construct the {\em reduced digraph}
$\redgraph$ with set of nodes $\{ 1,\ldots ,t\}$ setting an edge
$(\mu,\nu)$ whenever there exist $i\in N^{\mu}$ and $j\in N^{\nu}$
such that $(i,j)\in E$. We shall call a connected component
(or the corresponding set of nodes) of $\assgraph$ a {\em class} of $A$
and also use that term for the
nodes of $\redgraph$. Further, we also identify subsets $S$ of
nodes of $\redgraph$ with the union of the corresponding classes of
$A$, that is $S$ may denote $\cup_{\nu \in S}N^\nu$.

Each class $\mu$ is labeled by the corresponding maximal cycle
(geometric) mean $\alpha_{\mu }$,
which will be also called the {\em Perron root} of the class. We
write $\mu\rightarrow\nu$ if $\mu = \nu$ or if there exists a path
in $\redgraph$ connecting $\mu$ to $\nu$ (in other words, if $\mu$
has access to $\nu$). A set $I$ of classes is an {\em initial segment}
of $\redgraph$ if $\nu \in I$ and $\mu\rightarrow\nu$ imply
that $\mu \in I$. The set of all classes $\mu$ such that
$\mu\rightarrow\nu$ will be denoted by $\intl(\nu)$ and called the
{\em initial segment} generated by $\nu$ in $\redgraph$. If $S$ is a
set of classes, then a class $\nu \in S$ is said to be {\em initial}
in $S$ if $\mu\rightarrow \nu$ and $\mu \in S$ imply that $\mu =
\nu$. Similarly,  a class $\nu \in S$ is called {\em final} in $S$
if $\nu\rightarrow \mu$ and $\mu \in S$ imply that $\mu = \nu$. An
initial (resp. final) class in $\{1,\ldots ,t\}$ is simply called
initial (resp. final).
A class $\nu$ is said to be {\em spectral} if $\nu$ is initial, or
if $\ga_\nu > 0$ and $\mu\rightarrow\nu$ imply that
$\alpha_{\mu}\leq \alpha_{\nu}$. A spectral class $\nu$ is called
{\em premier spectral} if $\mu\rightarrow\nu$ and $\mu\neq \nu$
imply that $\alpha_\mu < \alpha_\nu$.

Access relations for $\assgraph$ and $\redgraph$ are normally
visualized in terms of a {\em Frobenius form}. There exists a
similarity permutation of $A$ such that
\[
A =
\begin{pmatrix}
A _{1 1}& 0 &\cdots & 0& 0\\
A _{2 1}& A _{2 2} &\cdots &0 & 0\\
\vdots &\vdots &\ddots & \vdots  & \vdots \\
A _{(t-1)1}&A _{(t-1)2} &\cdots &A _{(t-1)(t-1)}& 0 \\
A _{t 1} &A _{t 2} &\cdots &A _{t(t-1)}& A _{t t}
\end{pmatrix}
\]
with irreducible diagonal blocks $A _{\mu \mu}$ for $\mu=1,\ldots ,t$.

A Frobenius (normal) form of $A$ arises from each total ordering of
the classes of $\redgraph$ that is anti-compatible with the partial
order given by the access relations, viz. $\mu \rightarrow \nu$
implies $\mu \geq \nu$. In particular, given any initial segment $I$
of $\redgraph$ there is a Frobenius form of $A$ for which the
classes of $I$ are $r$, $r+1,\ldots ,t$ for some $r \in\{1,\ldots ,t\}$.

We now state the {\em fundamental spectral theorem of max algebra}.
Recall that the {\em support} of a vector $x \in \Rpn$
consists of all $i \in N$ such that $x_i > 0$.

\begin{theorem}\label{t:fundspec}
Let $A \in \Rpnn$ and $\gl \in \Rp$.
Then, a subset $U$ of $N$ is the support of an
eigenvector associated with $\gl$ if, and only if,
\begin{enumerate}[(i)]
\item\label{t:fundspec1}
There is an initial segment $I$ of $\redgraph$ such that
$U = \cup_{\nu \in I}N^\nu$,
\item\label{t:fundspec2}
All final classes $\nu$ in $I$ are spectral and satisfy
$\ga_\nu =\gl$.
\end{enumerate}
\end{theorem}

This theorem has a long history and has been stated in different ways,
see e.g.~\cite{GP-97,Gau-92,CG:79,BCG-09,But-10,BCOQ}.
The statement in Theorem~\ref{t:fundspec} is essentially the same
as the one that appeared in~\cite{GP-97}.

The following corollary is immediate.

\begin{corollary}\label{c:initnodes}
Let $A \in \Rpnn$. Then,
\begin{enumerate}[(i)]
\item\label{c:initnodes1}
$\gl$ is an eigenvalue if, and only if, there is a spectral class
$\nu$ such that $\ga_\nu = \gl$;
\item\label{c:initnodes2}
$\nu$ is a spectral class if, and only if, there exists an
eigenvector with support $\intl(\nu)$;
\item\label{c:initnodes3}
A spectral class $\nu$ is premier spectral if, and only if,
any eigenvector associated with $\ga_{\nu}$ whose support is contained
in $\intl(\nu)$ has its support equal to $\intl(\nu)$;
\item\label{c:initnodes4}
If the reduced digraph of $A$ has a unique spectral class $\nu$ with
Perron root $\ga_{\nu}$, then any eigenvector associated with
$\ga_{\nu}$ has support $\intl(\nu)$;
\item\label{c:initnodes5}
If the Perron roots of all classes are distinct, then all spectral
classes are premier spectral and all eigenvectors have support
$\intl(\nu$) for some spectral class $\nu$.
\end{enumerate}
\end{corollary}

The following well-known corollary also
follows easily from Theorem~\ref{t:fundspec}.

\begin{corollary}\label{c:posevec}
For any $A\in \Rpnn$ with $\gl(A)> 0$
the following statements are equivalent:
\begin{enumerate}[(i)]
\item\label{c:posevec1}
$A$ has a positive eigenvector.
\item\label{c:posevec2}
The Perron root of any final class is $\gl(A)$
(and so, in particular, all final classes are spectral).
\end{enumerate}
If either condition holds, then any positive eigenvector
is associated with the eigenvalue $\gl(A)$.
\end{corollary}

The proof of our next lemma essentially repeats arguments used to
prove Corollary~\ref{c:commoneig} and Theorem~\ref{t:cv-rmats}.

\begin{lemma}\label{l:commutant}
Let $A \in \R_+^{n\times n}$ and $C\in\R_+^{m\times m}$. If $AX = XC$,
where $X\in\R_+^{n\times m}$ and every column of $X$ is nonzero,
then any eigenvalue of $C$ is also an eigenvalue of $A$.
\end{lemma}

\begin{proof} Suppose that $\lambda \in \Lambda(C)$ and let
$z\in\R_+^m $ be an eigenvector of $C$ associated with $\lambda $. Then,
$A X z = X C z = \lambda X z$.
Since every column of $X$ is nonzero,
we have $X z \neq 0$ and thus $\lambda \in \Lambda(A)$.
\end{proof}

If $A$ and $C$ are irreducible, then in
Lemma~\ref{l:commutant} it is enough to assume that $X$ is
nonzero because the vector $z$ in the proof above is positive.
Thus, we obtain:

\begin{lemma}\label{l:icommutant}
Let $A \in \R_+^{n\times n}$ and $C\in\R_+^{m\times m}$
be irreducible matrices. If $AX = XC$, where $X\in\R_+^{n\times m}$
is nonzero, then $\lambda(A) = \lambda(C)$.
\end{lemma}

The following important lemma indicates what happens if a matrix
commutes with an irreducible matrix.

\begin{lemma}\label{l:uniqueevalue}
If $A,B\in\R_+^{n\times n}$ commute and $B$ is irreducible, then
\begin{enumerate}[(i)]
 \item\label{l:uniqueevalue1}
The Perron root of every final class and every initial class of $A$
is $\gl(A)$ (and so, in particular, all final classes are spectral);
\item\label{l:uniqueevalue2}
$A$ has the unique eigenvalue $\gl(A)$;
\item\label{l:uniqueevalue3}
If $A$ is reducible,
then at least two distinct classes of $A$ have Perron root $\gl(A)$.
\end{enumerate}
\end{lemma}
\begin{proof}
In the first place, note that the lemma is obvious if $\gl(A)=0$,
so we may assume that $\gl(A)> 0$.

\eqref{l:uniqueevalue1}
From Corollary~\ref{c:commoneig}, we know that $A$ and $B$
have a common eigenvector. Since $B$ is irreducible,
all its eigenvectors are positive.
It follows by~\eqref{c:posevec2} of
Corollary~\ref{c:posevec} that all final classes of $A$ have
Perron root $\gl(A)$ and are therefore spectral.
Similarly, the transpose $A^T$ commutes with the irreducible matrix $B^T$ and
therefore all final classes of $A^T$ have Perron root $\gl(A)$.
But the final classes of $A^T$ are precisely the initial classes of $A$.

\eqref{l:uniqueevalue2} This follows easily from~\eqref{l:uniqueevalue1},
Theorem~\ref{t:fundspec} and the definition of spectral class.

\eqref{l:uniqueevalue3} If $A$ is reducible,
either it has two initial classes or
an initial class and a distinct final class.
\end{proof}

\begin{remark}\label{r:spineqs}
{\rm In Corollary~\ref{c:spineqs}, the irreducibility assumption can
be relaxed. We need there that just {\em one} of the matrices is
irreducible, for then by~\eqref{l:uniqueevalue2} of Lemma~\ref{l:uniqueevalue}
each matrix has a unique eigenvalue.}
\end{remark}

The {\em transitive closure} of $\redgraph$ is the digraph
$\redgraph^*$ which has the edge $(\mu,\nu)$ if, and only if,
$\mu\rightarrow\nu$ in $\redgraph$. We shall say that $\nu$
{\em covers} $\mu$ in $\redgraph^*$ if $ \nu \neq \mu$,
$\nu \rightarrow \mu$ and the following property is satisfied:
$\nu \rightarrow \delta \rightarrow \mu$
implies that either $\delta=\mu$ or $\delta = \nu$.

The main result of this section is the following theorem.

\begin{theorem}\label{t:distroots}
Suppose that $A_1,\ldots ,A_r\in\R_+^{n\times n}$ pairwise commute
and that all classes of $A_i$, for each $i\in \{1,\ldots,r\}$, have
distinct Perron roots. Then,
\begin{enumerate}[(i)]
\item\label{t:d1}
All classes of $A_1,\ldots,A_r$ and $A_1+\cdots+ A_r$ coincide;
\item\label{t:d2}
The transitive closures of the reduced digraphs of
$A_1,\ldots,A_r$ and $A_1+\cdots+ A_r$ coincide;
\item\label{t:d3}
The spectral classes of the reduced digraphs of $A_1,\ldots,A_r$ and
$A_1+\cdots+ A_r$ coincide. In particular, $A_1,\ldots,A_r$ have the
same number of distinct eigenvalues;
\item\label{t:d4}
Let $\mu_1,\ldots,\mu_m$  be the common spectral classes of
$A_1,\ldots,A_r$ and denote the Perron root of the $\mu_j$-th class
of $A_i$ by $\ga^j_i$. Then, for any max polynomial
$p(x_1,\ldots,x_r)$, the eigenvalues of $p(A_1,\ldots,A_r)$ are
precisely $p(\alpha^j_1,\ldots,\alpha^j_r)$ for $j=1,\ldots,m$
(possibly with repetitions).
\end{enumerate}
\end{theorem}

\begin{proof}
\eqref{t:d1} Suppose that $C:= A_1+\cdots+ A_r$ is in Frobenius
form and partition $A_i$, for $i=1,\ldots,r$, correspondingly.
Evidently, a Frobenius  form of $B:=A_i$, for $i=1,\ldots,r$,
is a refinement of the Frobenius form of $C$. Since $B_{\mu \mu}$
and $C_{\mu \mu}$ commute and $C_{\mu \mu}$ is irreducible,
by~\eqref{l:uniqueevalue3} of Lemma~\ref{l:uniqueevalue} and our assumption,
it follows that $B_{\mu \mu}$ is also irreducible. Therefore,
$B=A_i$ is also in Frobenius form. This proves~\eqref{t:d1}.

\eqref{t:d2} Now suppose that $\nu$ covers $\mu$ in
the reduced digraph associated with $C$.
Then, for $B:=A_i$ the matrices
\[
\begin{pmatrix}
B_{\mu \mu} &  0 \\
B_{\nu \mu} &  B_{\nu \nu}
\end{pmatrix}
\enspace {\rm and }\enspace
 \begin{pmatrix}
C_{\mu \mu} &  0 \\
C_{\nu \mu} &  C_{\nu \nu}
\end{pmatrix}
\]
commute and, by assumption, $C_{\nu \mu} \neq 0$. Suppose that
$B_{\nu \mu} = 0$. Examining the $(2,1)$ block of the products of
these matrices we obtain
\begin{equation}
B_{\nu \nu}C_{\nu \mu} = C_{\nu \mu} B_{\mu \mu} \enspace .
\end{equation}
Since $B_{\mu \mu}$ and $B_{\nu \nu}$ are irreducible, it follows
from Lemma~\ref{l:icommutant} that the Perron roots of $B_{\mu \mu}$
and $B_{\nu \nu}$ are equal. This contradicts our assumption and
hence $B_{\nu \mu} \neq 0$. But two transitive digraphs coincide if
the cover relations are identical. This proves~\eqref{t:d2}.

\eqref{t:d3}
In the first place, observe that any initial segment $\intl(\nu)$
generated by a class $\nu$ in the reduced digraph associated with one
of the matrices $A_1,\ldots,A_r$ or $A_1+\cdots+ A_r$ is independent of the
choice of the matrix
because the transitive closures of their reduced digraphs coincide.
For this reason, in what follows we shall denote by $\intl(\nu)$ this
common initial segment and we shall not specify the matrix it corresponds to.

Let $\mu_j$ be a spectral class of $A_i$. Since all
classes of $A_i$ have distinct Perron roots,
from~\eqref{c:initnodes5} of Corollary~\ref{c:initnodes} it follows
that every spectral class is premier spectral and that every
eigenvector of $A_i$ associated with $\ga^j_i$ has support
$\intl(\mu_j)$. But, by Theorem~\ref{t:cv-rmats}, there are
eigenvalues of $A_k$ for $k\neq i$ that share an eigenvector
with the eigenvalue $\ga^j_i$ of $A_i$. Since this eigenvector has
support $\intl(\mu_j)$, by~\eqref{c:initnodes2} of
Corollary~\ref{c:initnodes} it follows that $\mu_j$
is a spectral class for all $A_k$.

Note that the above argument shows that any spectral class of $A_i$ is also
a spectral class of $A_1+\cdots+ A_r$. To prove the converse in max
algebra, suppose that $\mu$ is a spectral class of $A_1+\cdots + A_r$.
Using the additivity of Perron roots
(see Corollary~\ref{c:spineqs}), we obtain
\begin{equation}\label{e:maxinequ}
\oplus_{i=1}^r \lambda((A_i)_{\nu \nu}) =
\lambda((\oplus_{i=1}^r A_i)_{\nu \nu}) \leq
\lambda((\oplus_{i=1}^r A_i)_{\mu \mu}) =
\oplus_{i=1}^r \lambda((A_i)_{\mu \mu}) \enspace ,
\end{equation}
for all $\nu \in \intl(\mu)$. Without loss of generality,
assume that
$\oplus_{i=1}^r \lambda((A_i)_{\mu \mu})= \lambda((A_1)_{\mu \mu})$.
Then, from~\eqref{e:maxinequ} it follows that
$\lambda((A_1)_{\nu \nu})\leq \lambda((A_1)_{\mu \mu})$
for all $\nu \in \intl(\mu)$,
implying that $\mu$ is a spectral class of $A_1$,
and hence of all $A_i$.

\eqref{t:d4} By Theorem~\ref{t:cv-rmats},
for each common spectral class $\mu_j$ of $A_1,\ldots ,A_r$
there exists a common eigenvector $v^j$ which has support $\intl(\mu_j)$.
Since $A_i v^j =\alpha_i^j v^j$ for $i=1,\ldots , r$,
it follows that
$p(A_1,\ldots ,A_r)v^j = p(\alpha_1^j,\ldots,\alpha_r^j) v^j$ and thus
$p(\alpha_1^j,\ldots,\alpha_r^j)$ is an eigenvalue of
$p(A_1,\ldots,A_r)$. Let now $\lambda $ be an eigenvalue of
$p(A_1,\ldots,A_r)$. As $p(A_1,\ldots,A_r)$ commutes with $A_i$ for
all $i=1,\ldots,r$, by Theorem~\ref{t:cv-rmats} there exists an
eigenvector $v$ of $p(A_1,\ldots,A_r)$ associated with $\lambda $
which is also an eigenvector of $A_i$ for all $i$. Then,
by~\eqref{c:initnodes5} of Corollary~\ref{c:initnodes} there exists
a common spectral class $\mu_j$ of $A_1,\ldots ,A_r$ such that the
support of $v$ is equal to $\intl(\mu_j)$. Therefore, we have
$A_i v=\alpha_i^j v$ for all $i=1,\ldots,r$, implying that
$\lambda=p(\alpha_1^j,\ldots,\alpha_r^j)$ because
$\lambda v=p(A_1,\ldots,A_r)v=p(\alpha_1^j,\ldots,\alpha_r^j)v$.
\end{proof}

As it was already observed, under the assumptions of
Theorem~\ref{t:distroots}, the eigenvalues $\alpha^j_i$,
$i = 1,\ldots, r$, of the matrices $A_1,\ldots ,A_r$ are associated with
some common spectral class $\mu_j$ of their reduced digraphs.
We next show how to compute the intersection
of the corresponding eigencones.
Let $I$ be the initial segment generated by the spectral class
$\mu_j$ in any of the reduced digraphs associated with the matrices
$A_i$ (recall that this initial segment is independent
of the choice of the matrix because the transitive closures of
their reduced digraphs coincide).
We write uniquely each vector $x \in \R^n_+$ as $x[I] + x[I']$,
where $I'$ is the complement of $I$ in $\{1,\ldots ,n\}$.
Since $I$ is an initial segment of
the reduced digraphs associated with all the matrices $A_i$, there
is a Frobenius form of all these matrices such that $I = \{r,r+1,\ldots,t\}$
for some $r\in\{1,\ldots,t\}$.
If we denote the submatrix of $A_i$
based on the set of classes $I$ by $A_i[I,I]$,
then as the matrices $A_1,\ldots ,A_r$ commute in pairs,
it follows that also the matrices $A_1[I,I],\ldots ,A_r[I,I]$
commute in pairs. Therefore,
we can apply the method described in Section~\ref{s:intersection}
to compute the intersection of their principal eigencones.
Moreover, by Corollary~\ref{c:initnodes} we know that
$x\in V(A_1,\alpha_1^j) \cap \cdots \cap V(A_r,\alpha_r^j)$ if,
and only if, $x[I']=0$ and
$x[I]\in V(A_1[I,I],\alpha_1^j) \cap \cdots \cap V(A_r[I,I],\alpha_r^j)$,
where the latter is the intersection of the principal eigencones
of $A_i[I,I]$, because by the definition of these matrices
we have $\lambda(A_i[I,I])=\alpha_i^j$ for all $i = 1,\ldots, r$.

\section{Common scaling and application of Boolean algebra}\label{s:bmaxalg}

\subsection{Common scaling and saturation digraphs}

The whole of this section is in max algebra only. It is inspired by
the works of Cuninghame-Green and Butkovi\v{c}~\cite{CGB-08,But-10},
where commuting matrices are studied in the context of two-sided
systems and generalized eigenproblem. In these works, commuting
irreducible matrices are {\em assumed} to have a common eigennode.
We are going to show that it is always the case.

If $A=(a_{ij})$ and $B=(b_{ij})$ are irreducible and $AB=BA$,
then they have a common positive eigenvector $u$,
and using $U=\diag(u)$ they can be simultaneously scaled to
$\Tilde{A}:=U^{-1}AU$ and $\Tilde{B}:=U^{-1}BU$.
Assumed that $\lambda(A)=\lambda(B)=1$,
we obtain for $\Tilde{A}=(\Tilde{a}_{ij})$ and
$\Tilde{B}=(\Tilde{b}_{ij})$ that
\begin{equation}
\begin{split}
Au=u \Rightarrow & \forall i\exists j\colon
a_{ij}u_j=u_i\Leftrightarrow
\Tilde{a}_{ij}=1 \enspace ,\\
& \forall i,j\colon a_{ij}u_j\leq u_i \Leftrightarrow
\Tilde{a}_{ij}\leq 1 \enspace .\\
Bu=u \Rightarrow & \forall i\exists j\colon
b_{ij}u_j=u_i\Leftrightarrow
\Tilde{b}_{ij}=1 \enspace ,\\
& \forall i,j\colon b_{ij}u_j\leq u_i \Leftrightarrow
\Tilde{b}_{ij}\leq 1 \enspace .
\end{split}
\end{equation}
Defining $\Tilde{A}^{[1]}=(\Tilde{a}^{[1]}_{ij})$ and
$\Tilde{B}^{[1]}=(\Tilde{b}^{[1]}_{ij})$ by:
\begin{equation}
\Tilde{a}^{[1]}_{ij}=
\begin{cases}
1 & \text{if }\Tilde{a}_{ij}=1 \; ,\\
0 & \text{otherwise.}
\end{cases}
\enspace  \enspace
\Tilde{b}^{[1]}_{ij}=
\begin{cases}
1 & \text{if }\Tilde{b}_{ij}=1 \; ,\\
0 & \text{otherwise.}
\end{cases}
\end{equation}
we obtain that
\begin{equation}
\label{atilde-btilde} \forall i\exists j\colon
\Tilde{a}_{ij}^{[1]}=1 \enspace ,\quad \forall i\exists k\colon
\Tilde{b}_{ik}^{[1]}=1 \enspace .
\end{equation}
This means that both $\Tilde{A}^{[1]}$ and $\Tilde{B}^{[1]}$ are
incidence matrices of digraphs $\assgraph_1=(N,E_1)$ and
$\assgraph_2=(N,E_2)$, where each node has a nonzero number of
outgoing edges. These digraphs are the {\em saturation digraphs} of
$u$ with respect to $A$ and $B$, meaning that $(i,j)\in E_1$
(resp. $(i,j)\in E_2$) if, and only if, $a_{ij}u_j=u_i$
(resp. $b_{ij}u_j=u_i$). We recall the following well-known result.

\begin{proposition}[Baccelli et al.~\cite{BCOQ}]\label{sat-digraph}
Let $A\in\R_+^{n\times n}$ be irreducible and let $u\in\R_+^n$ be an eigenvector
of $A$. Then, the strongly connected components of the saturation
digraph of $u$ with respect to $A$ are the same as that of the critical digraph
$\crit(A)$.
\end{proposition}

This proposition tells us that the strongly connected components of
$\assgraph_1$ and $\assgraph_2$ are those of $\crit(A)$ and
$\crit(B)$.

\subsection{Commuting Boolean matrices}

Now we study in more detail the case of Boolean matrices, to show that two 
irreducible commuting matrices in max algebra always have a common eigennode.
In a similar way, the graphs of commuting Boolean matrices are studied 
in~\cite[Proposition 10]{DO-10a} to make an observation about the general case.

We need a couple of simple facts, which combined with 
Proposition~\ref{sat-digraph}, will provide the 
connection between max algebra and the Boolean case.

\begin{lemma}\label{l:011}
If the matrices $A,B\in\R_+^{n\times n}$ are such that
$a_{ij}\leq 1$ and $b_{ij}\leq 1$ for all $i,j\in N$,
then $(AB)^{[1]}=A^{[1]}B^{[1]}$.
\end{lemma}
\begin{proof}
For all $i,k\in N$, we may have two cases:
\begin{equation}\label{comm-choices}
\bigoplus_{j=1}^n a_{ij}b_{jk}=1 \enspace  \enspace \makebox{or } \enspace
\bigoplus_{j=1}^n a_{ij}b_{jk}<1 \enspace .
\end{equation}
In the first case of~\eqref{comm-choices}, there exists $h$
such that $a_{ih}b_{hk}=1$, which implies
$a_{ih}=b_{hk}=1$, since $a_{ij}\leq 1$ and
$b_{ij}\leq 1$ for all $i$ and $j$.
Passing to $A^{[1]}$ and $B^{[1]}$, we have
$a^{[1]}_{ih}=b^{[1]}_{hk}=1$ and thus 
$a^{[1]}_{ih}b^{[1]}_{hk}=1$. Using this we obtain
\begin{equation}
\label{comm-1mats1}
\bigoplus_{j=1}^n a^{[1]}_{ij}b^{[1]}_{jk}=1 \enspace .
\end{equation}
In the second case of~\eqref{comm-choices},
there are no such $h$ as above, and we obtain
\begin{equation}
\label{comm-1mats2} \bigoplus_{j=1}^n a^{[1]}_{ij}b^{[1]}_{jk}=0 \enspace .
\end{equation}
It follows that $(AB)^{[1]}=A^{[1]}B^{[1]}$.
\end{proof}

We immediately deduce the following observation.

\begin{lemma}\label{a1b1-comm}
If the matrices $A,B\in\R_+^{n\times n}$ are such that $AB=BA$ and
$a_{ij}\leq 1$, $b_{ij}\leq 1$ for all $i,j\in N$,
then $A^{[1]}B^{[1]}=B^{[1]}A^{[1]}$.
\end{lemma}

This motivates us to study the Boolean case in more detail.

\begin{theorem}\label{t:commdigr}
Let $\assgraph_1$ and $\assgraph_2$ be two commuting digraphs
(meaning that their incidence matrices commute)
with nonzero out-degree at each node,
and let $\assgraph_1^{\mu}=(N_1^{\mu},E_1^{\mu})$ for $\mu=1,\ldots,m_1$ and
$\assgraph_2^{\nu}=(N_2^{\nu},E_2^{\nu})$ for $\nu=1,\ldots,m_2$ be the
strongly connected components of $\assgraph_1$ and $\assgraph_2$ respectively. 
Then, there exists a cycle $c_1\in \assgraph_1$ such
that all nodes on this cycle belong to $\bigcup_{\nu =1}^{m_2}
N^{\nu}_2$, and a cycle $c_2\in \assgraph_2$ such that all nodes on
this cycle belong to $\bigcup_{\mu =1}^{m_1} N^{\mu}_1$.
\end{theorem}

\begin{proof}
Pick $\mu_1\in\{1,\ldots,m_1\}$ and consider the digraph
$\assgraph_2[N^{\mu_1}_1]$ induced by $N^{\mu_1}_1$. 
Then, either this induced digraph has a cycle and the claim is true, 
or it is acyclic. In the latter case, let $i\in N^{\mu_1}_1$ be a leaf in
$\assgraph_2[N_1^{\mu_1}]$, i.e. a node with no edges back into $N_1^{\mu_1}$.
Let $M=\{j\mid (i,j)\in E_2\}$.
As $i$ is a leaf in $\assgraph_2[N_1^{\mu_1}]$,
we have $M\cap N_1^{\mu_1}=\emptyset$.
There is a cycle $c\in\assgraph_1$ which goes through $i$.
Select $j\in M$ and consider the path $c\circ(i,j)$
(first turn around along $c$ in $\assgraph_1$ then move $i\to j$ in
$\assgraph_2$).
As the digraphs commute, there is a path $P=(i,k)\circ P'$ 
connecting node $i$ with node $j$ such that $(i,k)\in E_2$ 
and the path $P'\in\assgraph_1$ is of the same length as $c$.
Hence, for each node $j\in M$ there exists a node $k\in M$
such that $k$ has access to $j$ in $\assgraph_1$.
This implies that some nodes in $M$ lie on a cycle in $\assgraph_1$
so that $M$ intersects a component
$\assgraph_1^{\mu_2}$ of $\assgraph_1$.

Consider the digraph $\assgraph_2[N_1^{\mu_2}]$.
If it is not acyclic then the claim is true,
otherwise we take $j\in M\cap N_1^{\mu_2}$ and
proceed to a leaf $k$ accessed by $j$ in $\assgraph_2[N_1^{\mu_2}]$.
We have obtained a path from $i$ to
$k$ in $\assgraph_2$, whose nodes lie in $\bigcup_{\mu=1}^{m_1}
N^{\mu}_1$. Arguing as above we can continue this path until we
obtain a cycle $c_2$ in $\assgraph_2$ which has all its nodes in
$\bigcup_{\mu=1}^{m_1} N^{\mu}_1$. The cycle $c_1$ in $\assgraph_1$
which has all its nodes in $\bigcup_{\nu=1}^{m_2} N^{\nu}_2$ is
obtained analogously. The claim is proved.
\end{proof}

Theorem~\ref{t:commdigr} implies notable facts about
the critical digraphs of two irreducible commuting matrices in max algebra.

\begin{theorem}\label{t:commoneigennode}
If two irreducible matrices $A,B\in\R_+^{n\times n}$ commute, then the
conclusion of Theorem~\ref{t:commdigr} holds for the strongly connected components
of $\crit(A)$ and $\crit(B)$. In particular, $A$ and $B$ have a common eigennode.
\end{theorem}
\begin{proof}
If $A,B\in\R_+^{n\times n}$ commute, then they have a common eigenvector $u$ by
Corollary~\ref{c:commoneig}. If these matrices are irreducible, then $u$
is positive and $U:=\diag(u)$ can be used to make a simultaneous
diagonal similarity scaling: $\Tilde{A}:=U^{-1}AU$ and $\Tilde{B}:=U^{-1}BU$.
Evidently, $\Tilde{A}\Tilde{B}=\Tilde{B}\Tilde{A}$. Also, 
we have $\crit(\Tilde{A})=\crit(A)$ and $\crit(\Tilde{B})=\crit(B)$. 
Notice that $\Tilde{A}^{[1]}$, resp. $\Tilde{B}^{[1]}$, 
is the incidence matrix of the saturation digraph of $u$ 
with respect to $A$, resp. to $B$. These saturation digraphs will
be denoted by $\assgraph_1$ and $\assgraph_2$, respectively 
(with the intention to use Theorem~\ref{t:commdigr}). By Lemma~\ref{a1b1-comm},
we have $\Tilde{A}^{[1]}\Tilde{B}^{[1]}=\Tilde{B}^{[1]}\Tilde{A}^{[1]}$. 
As $\assgraph_1$ and $\assgraph_2$ are saturation digraphs, 
each node in these digraphs has a nonzero out-degree. 
Applying Theorem~\ref{t:commdigr}, we obtain that its conclusion
holds for the strongly connected components of
$\assgraph_1$ and $\assgraph_2$. 
By Proposition~\ref{sat-digraph}, these components
are precisely the strongly connected components of $\crit(A)$ and $\crit(B)$.
Now the conclusion of Theorem~\ref{t:commdigr} also implies that $A$ and $B$
have a common eigennode.
\end{proof}

Let us consider a special case, 
which usually appears if $A$ and $B$ are chosen at random.

\begin{corollary}
Let two irreducible matrices $A,B\in\R_+^{n\times n}$ commute. If $\crit(A)=(N^A_c,E^A_c)$
and $\crit(B)=(N^B_c,E^B_c)$ both consist of just one cycle, then
$N_c^A=N_c^B$.
\end{corollary}

\section{Examples of commuting matrices in max algebra}\label{s:examples}

In this section we give several examples in max algebra, which will
appear now as the semiring $(\R \cup \{ -\infty \}, \max ,+)$, i.e.
the set $\R \cup \{ -\infty \}$ equipped with $\max$ as ``addition''
and the usual sum as ``multiplication''.
This semiring is isomorphic to $(\Rp, \max ,\times )$
via the logarithmic transform.

Consider the irreducible commuting matrices
\[
A_1=\begin{pmatrix}
-2 & 1 & -\infty \\
-1 & -1 & -2 \\
-1 & -\infty & -2
\end{pmatrix}
\enspace {\rm and }\enspace A_2= \begin{pmatrix}
0 & -1 & -1 \\
-\infty & 0 & -4 \\
-3 & -\infty & 0
\end{pmatrix} \enspace .
\]
Then, it is straightforward to check that $\lambda (A_1)= \lambda(A_2)= 0$,
$N_c^{A_1}=\{ 1,2 \} $ and $N_c^{A_2}=\{ 1,2,3 \} $.
Therefore, as claimed in Theorem~\ref{t:commoneigennode}, $A_1$
and $A_2$ have a common eigennode.

In order to compute their common eigenvectors,
we apply the method described in Section~\ref{s:intersection}.
Since
\[
Q(A_1)=\begin{pmatrix}
0 & 1 & -1 \\
-1 & 0 & -2 \\
-1 & 0 & -2
\end{pmatrix}
\enspace {\rm and }\enspace Q(A_2)= \begin{pmatrix}
0 & -1 & -1 \\
-7 & 0 & -4  \\
-3 & -4 & 0
\end{pmatrix} \enspace ,
\]
it follows that
\[
Q(A_1)Q(A_2)=\begin{pmatrix}
0 & 1 & -1 \\
-1 & 0 & -2 \\
-1 & 0 & -2
\end{pmatrix} \enspace .
\]
Then, by~\eqref{qa1ar} we have
\[
V(A_1,0)\cap V(A_2,0)=V(Q(A_1)Q(A_2),0)=\left\{\lambda (1,0,0)^T
\mid \lambda \in \R \cup \{ -\infty \} \right\} \enspace .
\]

The following example of commuting matrices illustrates
Lemma~\ref{l:uniqueevalue}. Let
\[
A=
\begin{pmatrix}
1 & -\infty & -\infty \\
1 & 0 & -\infty \\
0 & 1 & 1  \\
\end{pmatrix}
\enspace \makebox{ and }\enspace
B=
\begin{pmatrix}
0 & 0 & 0 \\
0 & 0 & 0 \\
0 & 0 & 0  \\
\end{pmatrix} \enspace .
\]
Then, $A$ and $B$ commute, $B$ is irreducible and $A$ satisfies the
conditions of Lemma~\ref{l:uniqueevalue}.

As an example of reducible commuting matrices, consider
\[
A_1=
\begin{pmatrix}
0 & -\infty & -\infty &-\infty \\
1 & 3 & -\infty & -\infty \\
2 & -\infty & -1 & -\infty \\
-\infty & -\infty & 0 & 2
\end{pmatrix}
\enspace \makebox{ and }\enspace
A_2=
\begin{pmatrix}
6 & -\infty & -\infty & -\infty \\
5 & 7 & -\infty & -\infty \\
8 & -\infty & 5 & -\infty \\
5 & -\infty & 6 & 8
\end{pmatrix} \enspace .
\]
The classes of these matrices are their diagonal elements.
Since the Perron roots of the classes (i.e. the diagonal entries
in this case) of each of these matrices are distinct,
we know by Theorem~\ref{t:distroots} that the
transitive closure of the reduced digraph associated with these
matrices are the same, even if these digraphs are different, as can
be easily checked. By the same theorem, we know that the spectral
classes of the associated reduced digraph coincide. In this case, for
both matrices the spectral classes are $2$ and $4$. Each of these
matrices has two different eigenvalues corresponding to their
spectral classes. The eigenvalues of $A_1$ are $3$ and $2$ and the
ones of $A_2$ are $7$ and $8$.

\section{Classical nonnegative matrices}\label{s:classical}

In this section we assume knowledge of some  basic results on
nonnegative matrices found in e.g.~\cite{BP-94} or~\cite{Gan-59}. We
state analogs for nonnegative matrices in classical matrix algebra
of the results for matrices in max algebra proved in
Section~\ref{s:commoneig} (except for the last subsection) and in
Section~\ref{s:fnf}. The results and proofs are essentially
identical. We first need to redefine some symbols and terms.

Let $A \in \Rpnn$. Following standard terminology, we call an
eigenvalue $\mcm$ of $A$ a {\em distinguished eigenvalue} of $A$ if
$\mcm \geq 0$ and there is a nonnegative eigenvector corresponding
to it. In this section, $\Lambda(A)$ will be the set of
distinguished eigenvalues of $A$ and $V(A,\lambda)$ the convex cone
of nonnegative eigenvectors (and the $0$ vector) associated with a
distinguished eigenvalue $\lambda$. By the Perron-Frobenius theorem,
$\Lambda(A)$ is nonempty and the largest element in $\Lambda(A)$ is
called the {\em Perron root} of $A$. Moreover, any eigencone
$V(A,\lambda)$ is finitely generated, and the intersection of
finitely generated convex cones is again finitely generated.
Matrices leaving a cone invariant in $\Rpn$ (indeed in $\R^n$) have
been much studied, see e.g.~\cite{TS-07}. Proposition~\ref{p:common}
is well known in this context.

Lemma~\ref{l:InvEigenCone}, Theorem~\ref{t:commoneig},
Theorem~\ref{t:cv-rmats}, Corollary~\ref{c:commoneig} and their
proofs go through without further change to the classical
nonnegative case, except that we need to insert the adjective
``nonnegative'' in Corollary~\ref{c:commoneig}:

\begin{mycorollary:commoneigP}
If $A,B\in\R_+^{n\times n}$ commute, then they have a common
nonnegative eigenvector.
\end{mycorollary:commoneigP}

It follows that if $A$ and $B$ are commuting nonnegative matrices and
one of them is irreducible, then they have a common Perron vector.

Theorem~\ref{t:polynomials} and Corollary \ref{c:spineqs} are also
valid in the classical nonnegative case under the following
assumptions: the matrices $A_1,\ldots,A_r \in \Rpnn$ commute in
pairs and $p(x_1,\ldots,x_r)$ is a real polynomial such that
$p(A_1,\ldots,A_r)$ is nonnegative and, in the case of
Corollary~\ref{c:spineqs}, all the coefficients of $p(x_1,\ldots,x_r)$ are
nonnegative. In the latter, by the analog of Remark~\ref{r:spineqs},
we need to assume only one of the $A_i$ is irreducible.

Turning to Section~\ref{s:fnf}, we again construct the reduced
digraph of $A\in \Rpnn$ and we now label each class $\mu$ with its
classical Perron root $\ga_\mu$.  By a theorem of
Frobenius~\cite{Fro-12}, we replace Theorem~\ref{t:fundspec} by:

\begin{mytheorem:spectrnodesP}
Let $A \in \Rpnn$ and $\gl \in \Rp$.
Then, a subset $U$ of $N$ is the support of a nonnegative
eigenvector associated with $\gl$ if, and only if,
\begin{enumerate}[(i)]
\item\label{t:fundspec1P}
There is an initial segment $I$ such that $U = \cup_{\nu \in I}N^\nu$,
\item\label{t:fundspec2P}
All final classes $\nu$ in $I$ are premier spectral and satisfy
$\ga_\nu =\gl$.
\end{enumerate}
\end{mytheorem:spectrnodesP}

See e.g.~\cite{Sch-86}. We observe that the supports of nonnegative
eigenvectors of $A \in \Rpnn$  are completely determined in 
Theorem~\ref{t:fundspec}A by (i) the classes (i.e. the strongly connected
components) of $\assgraph$, (ii) the Perron roots of these classes
and (iii) the access relations of $\redgraph$ (equivalently the edges
of $\redgraph^*$). A similar remark holds for 
Theorem~\ref{t:fundspec} and other results in 
Sections~\ref{s:fnf} and~\ref{s:classical}.

We restate Corollary~\ref{c:initnodes} as:

\begin{mycorollary:initnodesP}
Let $A \in \Rpnn$. Then,
\begin{enumerate}[(i)]
\item\label{c:initnodes1P}
$\gl$ is a distinguished eigenvalue if, and only if, there is a
premier spectral class $\nu$ such that $\ga_\nu = \gl$;
\item\label{c:initnodes2P}
$\nu$ is a premier spectral class if, and only if, there exists a
nonnegative eigenvector with support $\intl(\nu)$;
\item\label{c:initnodes3P}
If $\nu$ is a premier spectral class, then any nonnegative
eigenvector associated with $\ga_{\nu}$ whose support is contained
in $\intl(\nu)$ has its support equal to $\intl(\nu)$;
\item\label{c:initnodes4P}
If the reduced digraph of $A$ has a unique premier spectral class
$\nu$ with Perron root $\ga_{\nu}$, then any nonnegative eigenvector
associated with $\ga_{\nu}$ has support $\intl(\nu)$;
\item\label{c:initnodes5P}
If the Perron roots of all classes are distinct, then all
nonnegative eigenvectors have support $\intl(\nu$) for some premier
spectral class $\nu$.
\end{enumerate}
\end{mycorollary:initnodesP}

The analog of Corollary~\ref{c:posevec} in nonnegative linear
algebra is well known, but we need to replace~\eqref{c:posevec2}
of Corollary~\ref{c:posevec} by:
``The Perron root of any final class is $\gl(A)$ and all final
classes are premier spectral''. Lemma~\ref{l:commutant}
goes through without change except that we need to replace
``eigenvalue'' with ``distinguished eigenvalue'' and
Lemma~\ref{l:icommutant} also holds in nonnegative linear algebra.

%
%
%

In the classical nonnegative case we obtain the following known
stronger form of Lemma~\ref{l:uniqueevalue} which may be found
on~\cite[p.53]{BP-94}.  We give a short proof along the lines of the
proof of Lemma~\ref{l:uniqueevalue}.

\begin{mylemma:uniqueevalueP}\label{l:uniqueevaluePrima}
If $A,B\in\R_+^{n\times n}$ commute and $B$ is irreducible,
then the Perron root of $A$ is its unique distinguished eigenvalue.
Moreover, if $A$ is reducible, it is completely reducible
(viz, the direct sum of irreducible matrices after a permutation similarity).
\end{mylemma:uniqueevalueP}

\begin{proof}
We repeat the proof of~\eqref{l:uniqueevalue1} of
Lemma~\ref{l:uniqueevalue} to show that both $A$ and $A^T$ have
positive eigenvectors. This implies that all initial and final classes
in the reduced digraph of $A$ are premier spectral with Perron root $\gl(A)$.
But a premier spectral class cannot have access to another premier spectral
class with the same Perron root. It follows that all initial classes are
final and vice versa. This means that a class has access only to itself,
which proves the lemma.
\end{proof}

Theorem~\ref{t:distroots} also holds in nonnegative algebra, with
exception of the last part of~\eqref{t:d3} whose proof is specific
to max algebra. Thus we obtain the following main theorem of this
section.

\begin{mytheorem:distrootsP}
Suppose that $A_1,\ldots ,A_r\in\R_+^{n\times n}$ pairwise commute
and that all classes of $A_i$, for each $i\in\{1,\ldots,r\}$, have
distinct Perron roots. Then,
\begin{enumerate}[(i)]
\item\label{t:d1pr}
All classes of $A_1,\ldots,A_r$ and $A_1+\cdots+ A_r$ coincide;
\item\label{t:d2pr}
The transitive closures of the reduced digraphs of
$A_1,\ldots,A_r$ and $A_1+\cdots+ A_r$ coincide;
\item\label{t:d3pr}
The reduced digraphs of $A_1,\ldots,A_r$ have the same premier
spectral classes, which are premier spectral classes of
$A_1+\cdots+ A_r$. In particular,
$A_1,\ldots,A_r$ have the same number of distinct distinguished eigenvalues;
\item\label{t:d4pr}
Let $\mu_1,\ldots,\mu_m$  be the common premier spectral classes of
$A_1,\ldots,A_r$ and denote the Perron root of the $\mu_j$-th class
of $A_i$ by $\ga^j_i$. Then, for any real polynomial
$p(x_1,\ldots,x_r)$ such that $p(A_1,\ldots,A_r)$ is nonnegative,
the  distinguished eigenvalues of $p(A_1,\ldots,A_r)$ are precisely
$p(\alpha^j_1,\ldots,\alpha^j_r)$ for $j=1,\ldots,m$ (possibly with
repetitions).
\end{enumerate}
\end{mytheorem:distrootsP}

We end this section with an example to illustrate Theorem~\ref{t:distroots}A.
\begin{example} {\rm
Let
\[
A =
\begin{pmatrix}
    10  &   0 &    0\\
     5  &   0 &    0\\
     2  &   3 &    3
\end{pmatrix} \enspace \makebox{ and } \enspace
B =
\begin{pmatrix}
     3 &    0 &  0\\
     1 &    1 &  0\\
     0 &    1 &   2
\end{pmatrix} \enspace .
\]
Then, $AB = BA$. The classes of $A$ and $B$ are their diagonal
elements, and the skeleton of their reduced digraphs (meaning the
diagram of cover relations) is
\[
1 \leftarrow 2 \leftarrow 3\; .
\]
The premier spectral classes of both matrices are $1$ and $3$ and
the distinguished eigenvalues are the corresponding entries. Their
common (nonnegative) eigenvectors are $(2, 1, 1)^T$ and
$(0, 0, 1)^T$, respectively.

Of course, $A^T$ and $B^T$ also commute. Note that the skeleton of
their reduced digraphs is obtained by reversing the arrows in the
diagram above. The only spectral class of $A^T$ or $B^T$ is 1 and
their common eigenvector is $(1, 0, 0)^T$.

We also observe that $p(A,B) = A^2B - AB \geq 0$ satisfies the
conditions of~\eqref{t:d4pr} of Theorem~\ref{t:distroots}A.
}
\end{example}

{\em Acknowledgement.}\/ We thank M. Drazin, T. Hawkins and R. Horn
for comments which have helped to improve this paper.
P. Butkovi\v{c} and B.~S. Tam deserve particular thanks for their
careful reading of our manuscript and many suggestions.


\begin{thebibliography}{10}

\bibitem{ABG-06}
M.~Akian, R.~Bapat, and S.~Gaubert.
\newblock{\em Max-Plus Algebra}.
\newblock Chapter 25 in the Handbook of Linear Algebra, L. Hogben, R. Brualdi, A. Greenbaum, and R. Mathias (editors), Discrete Mathematics and Its Applications, Volume 39, Chapman and Hall, 2006.

\bibitem{BCOQ}
F.~L. Baccelli, G.~Cohen, G. J. Olsder, and J. P. Quadrat.
\newblock {\em Synchronization and Linearity: an Algebra for Discrete Event
  Systems}.
\newblock Wiley, 1992.

\bibitem{BP-94}
\newblock A.~Berman and R.~Plemmons.
\newblock {\em Nonnegative matrices in the mathematical sciences,
2nd Edn}.
\newblock SIAM, 1994

\bibitem{But-10}
P.~Butkovi{\v{c}}.
\newblock {\em Max-linear systems: theory and applications}.
\newblock Springer, 2010 (in print).

\bibitem{BCG-09}
P.~Butkovi{\v{c}}, R.~A. Cuninghame-Green, and S.~Gaubert.
\newblock {\em Reducible spectral theory with applications to the robustness of matrices in max-algebra}.
\newblock SIAM J. on Matrix Analysis and Appl., 31(3):1412-1431, 2009.

\bibitem{BH-84}
P.~Butkovic and G.~Hegedus.
\newblock {\em An elimination method for finding all solutions of the system of linear equations over an extremal algebra}.
\newblock Ekonom.-Mat. Obzor (Prague), 20(2):203--215, 1984.

\bibitem{Cay-58}
A. Cayley.
\newblock {\em A memoir on the theory of matrices}. Phil. Trans.
Royal. Soc, 148:17--37 (1858),
\newblock Coll. Math. Papers, 2:475-496, Cambridge (1889).

\bibitem{CTGG-99}
J. Cochet-Terrasson, S. Gaubert and J. Gunawardena.
\newblock {\em A constructive fixed point theorem for min-max functions}.
\newblock Dynamics and Stability of Systems, 14(4):407--43, 1999.

\bibitem{CG:79}
R.~A. Cuninghame-Green.
\newblock {\em Minimax Algebra}, volume 166 of {\em Lecture Notes in Economics and Mathematical Systems}.
\newblock Springer, Berlin, 1979.

\bibitem{CGB-08}
R.~A. Cuninghame-Green and P.~Butkovi{\v{c}}.
\newblock {\em Generalised eigenproblem in max algebra}.
\newblock In {\em Proceedings of the 9th International Workshop WODES 2008},
pages 236--241, 2008.
\newblock http://ieeexplore.ieee.org/stamp.

\bibitem{DO-10a}
D.~Dol\v{z}an and P.~Oblak.
\newblock {\em Commuting graphs of matrices over semirings}.
\newblock Linear Algebra Appl. (to appear).

\bibitem{DO-10b}
D.~Dol\v{z}an and P.~Oblak.
\newblock {\em Noncommuting graphs of matrices over semirings}.
\newblock Linear Algebra Appl. (to appear).

\bibitem{Dra-51}
M.~P.~Drazin.
\newblock {\em Some generalizations of matrix commutativity}.
\newblock Proc. of the London Math. Soc., 3(1):222--231, 1951.

\bibitem{DDG-51}
M.~P.~Drazin, J.~W.~Dungey and K.~W.~Gruenberg.
\newblock {\em Some theorems on commutative matrices}.
\newblock J. London Math. Soc. 26:221--228, (1951).


\bibitem{Fro-78}
F.~G.~Frobenius.
\newblock {\em \"Uber lineare Substitutionen und bilineare Formen}.
\newblock Journal f\"ur die reine und angewandte Mathematik, 84:1--63 (1878),
\newblock Ges. Abh., 1:343-405, Springer.

\bibitem{Fro-96}
F.~G.~Frobenius.
\newblock {\em \"Uber vertauschbare Matrizen}.
\newblock Sitzunsb. K\"on. Preu{\ss}. Akad. Wiss. Berlin,
(1896) 601--614, Ges. Abh., 2:705--718, Springer.

\bibitem{Fro-08}
F.~G.~Frobenius.
\newblock {\em \"Uber Matrizen aus positiven Elementen}.
\newblock Sitzungsber. K\"on. Preuss. Akad. Wiss. Berlin,
(1908) 471--476,
\newblock Ges. Abh., 3:404-409, Springer, Berlin.

\bibitem{Fro-09}
F.~G.~Frobenius.
\newblock {\em \"Uber Matrizen aus positiven Elementen, II}.
\newblock Sitzungsber. K\"on. Preuss. Akad. Wiss. Berlin, (1909) 514--518.
\newblock Ges. Abh., 3:410-414, Springer, 1968.

\bibitem{Fro-12}
F.~G.~Frobenius.
\newblock {\em \"Uber Matrizen aus nicht negativen Elementen}.
\newblock Sitzungsber. K\"on. Preuss. Akad. Wiss. Berlin, (1912) 456--477,
\newblock  Ges.~Abh 3, 546--567. Springer, 1968.

\bibitem{Gan-59}
F.~R.~Gantmacher.
\newblock {\em The Theory of Matrices}.
\newblock Chelsea, 1959.

\bibitem{Gau-92}
S.~Gaubert.
\newblock {\em Th\'{e}orie des sytemes lin\'{e}aires dans les dio\"{\i}des}.
\newblock Th\`{e}se, \'{E}cole des Mines des Paris, 1992.

\bibitem{GK09}
S. Gaubert and R. D. Katz.
\newblock {\em Minimal half-spaces and external representation of tropical polyhedra}.
\newblock Journal of Algebraic Combinatorics (to appear), e-print arxiv:0908.1586, 2009.

\bibitem{GP-97}
S. Gaubert and M. Plus.
\newblock {\em Methods and applications of (max,+) linear algebra}.
\newblock 14th annual symposium on theoretical aspects of computer
science (STACS), number 1200 in LNCS, 261--282, Springer, 1997.

\bibitem{HJ-85}
R.~A. Horn and C.~R. Johnson.
\newblock {\em Matrix Analysis}.
\newblock Cambridge, 1985.

\bibitem{HOW-05}
B.~Heidergott, G. J. Olsder, and J.~van~der Woude.
\newblock {\em Max-plus at Work}.
\newblock Princeton, 2005.

\bibitem{KM:97}
V.~N. Kolokoltsov and V.~P. Maslov.
\newblock {\em Idempotent analysis and its applications}.
\newblock Kluwer, 1997.

\bibitem{Lit-07}
G.~L. Litvinov.
\newblock {\em The {M}aslov dequantization, idempotent and tropical mathematics: a brief introduction}.
\newblock Journal of Math. Sciences, 140(3):426--444, 2007.
\newblock E-print arXiv:math.GM/0507014.

\bibitem{MS-10}
T.~Mc Kinley and B.~Shekhtman.
\newblock {\em On simultaneous block-diagonalization of cyclic sequences
of commuting matrices}.
\newblock Lin. Multilin. Algebra, 58:245--256, 2010.

\bibitem{Per-07}
O.~Perron.
\newblock {\em Grundlagen f\"ur eine Theorie des Jacobischen
Kettenbruchalgorithmus}.
\newblock Math. Ann. 63:1--76, 1907.

\bibitem{Per-07a}
O.~Perron.
\newblock {\em Zur Theorie der Matrices}.
\newblock Math. Ann. 6, 248--263, 1907.

\bibitem{Rad-99}
H.~Radjavi.
\newblock {\em The Perron-Frobenius theorem revisited}.
\newblock Positivity 3, 317--331, 1999.

\bibitem{RR-00}
H.~Radjavi and P. Rosenthal.
\newblock {\em Simultaneous triangularization}.
\newblock Springer, 2000.

\bibitem{Rot-07}
U.~G. Rothblum.
\newblock{\em Nonnegative and stochastic matrices}.
\newblock  Handbook for Linear Algebra, ed. L. Hogben, Article 9, Chapman and Hall (2007).

\bibitem{Sch-86}
H.~Schneider.
\newblock {\em The influence of the marked reduced graph of a nonnegative matrix on the Jordan form and related properties: A survey}.
\newblock Linear Algebra Appl., 84:161--189, 1986.

\bibitem{Schu-02}
I.~Schur.
\newblock {\em \"Uber einen Satz aus der Theorie de vertauschbaren Matrizen}.
\newblock Sitzunsb. Preu{\ss}. Akad. Wiss. Berlin, Phys-Math. Klass, 120--125 (1902),
\newblock {\em Ges. Abh. }1:73--78, Springer.


\bibitem{TS-07}
B.~S.~Tam and H.~Schneider.
\newblock {\em Matrices leaving a cone invariant}.
\newblock Handbook for Linear Algebra, ed. L. Hogben, Article 26, Chapman and Hall (2007).

\bibitem{Wey-90}
E.~Weyr.
\newblock {\em Zur Theorie der bilinearen Formen}.
\newblock Monatsh. Math. Physik 1:163--236, 1890.

\bibitem{Yak-90}
S.~Yakovenko.
\newblock {\em On the concept of an infinite extremal in stationary problems
of dynamic optimization}.
\newblock Soviet Math. Dokl., 40(2):384--388, 1990.

\end{thebibliography}

\end{document}